\documentclass[12pt,a4paper]{amsart}
\usepackage{enumerate}
\usepackage{amssymb,latexsym}
\usepackage{amsfonts}
\usepackage{graphicx}
 \newtheorem{theorem}{Theorem}[section]
 
 \newtheorem{lemma}[theorem]{Lemma}
 \newtheorem{proposition}[theorem]{Proposition}
 \theoremstyle{definition}
 \newtheorem{definition}[theorem]{Definition}
 \newtheorem{properties}[theorem]{Properties}
 \theoremstyle{remark}

 \numberwithin{equation}{section}

\begin{document}

\title{A New Product Formula Involving Bessel Functions}

\author[M.A. Boubatra]{Mohamed Amine Boubatra}
\address{%
Universit\'e de Tunis El Manar, Facult\'e des Sciences de Tunis
\\ LR11ES11 Laboratoire d'Analyse Math\'ematiques et Applications,
2092 Tunis, Tunisie}
\email{boubatra.amine@yahoo.fr}
\author[S. Negzaoui]{Selma Negzaoui}
\address{%
Universit\'e de Tunis El Manar, Facult\'e des Sciences de Tunis
\\ LR11ES11 Laboratoire d'Analyse Math\'ematiques et Applications,
2092 Tunis, Tunisie} \email{selma.negzaoui@fst.utm.tn}
%--------------------------------------------
\author[M. Sifi]{Mohamed Sifi}
\address{%
Universit\'e de Tunis El Manar, Facult\'e des Sciences de Tunis
\\ LR11ES11 Laboratoire d'Analyse Math\'ematiques et Applications,
2092 Tunis, Tunisie } \email{mohamed.sifi@fst.utm.tn}
\renewcommand{\thefootnote}{}

\footnote{keywords : Product formula ; Bessel functions ; Gegenbauer polynomials ; generalized Hankel and Dunkl transforms ; convolution structure ; Translation operator}

\begin{abstract}
In this paper, we consider the normalized Bessel function of index $\alpha > -\frac{1}{2}$, we find an integral representation of the term $x^nj_{\alpha+n}(x)j_\alpha(y)$. This allows us to establish a product formula for the generalized Hankel function $B^{\kappa,n}_\lambda$ on $\mathbb{R}$. $B^{\kappa,n}_\lambda$ is the kernel of the integral transform $\mathcal{F}_{\kappa,n}$ arising from the Dunkl theory. Indeed we show that
$B^{\kappa,n}_\lambda(x)B^{\kappa,n}_\lambda(y)$ can be expressed as an integral in terms of $B^{\kappa,n}_\lambda(z)$ with
explicit kernel invoking Gegenbauer polynomials for all $n\in\mathbb{N}^\ast$. The obtained result generalizes the product formulas proved by M. R\"osler for Dunkl kernel when n=1 and by S. Ben Said when $n=2$. \\ As application, we define and study a translation operator and a convolution structure associated to $B^{\kappa,n}_\lambda$. They share many important properties with their analogous in the classical Fourier theory.
 %\PACS{PACS code1 \and PACS code2 \and more}
 %\subclass
\end{abstract}
\maketitle
\section{Introduction}
In the theory of classical eigenfunctions of singular Sturm-Liouville equations such as the
Jacobi, Laguerre and Hermite polynomials or the Bessel, Whittaker and Jacobi functions,
integral representations of products are very deep formulas which prove to be useful in various
branches of mathematics. One can see \cite{As,Er1,Er2,Wa} for a comprehensive list of such identities.
Such product formulas are used, for instance, to introduce a translation operator and a convolution structure.

The purpose of this paper is to build a translation operator and a convolution structure for the one-dimensional generalized Fourier transform
${\mathcal{F}}_{\kappa,n}$, with $n\in \mathbb{N}^\star$ and $\kappa >\displaystyle \frac{n-1}{2n}$, arising from the Dunkl's theory \cite{BKO1,Du}. Indeed, it is well known that the Euclidean Fourier transform $\mathcal{F}$ is defined by
$$\mathcal{F}f(\xi) =\frac{1}{(2\pi)^{d/2}}\int_{\mathbb{R}^d}
f(x)\,K(x,\xi) dx,$$
where the kernel $K(x, \xi)$ is the unique solution of the partial differential equation system 
$$\left\{\begin{array}{ll}
 \partial x_j\,K(x, \xi) =
-i\xi_jK(x, \xi) &\;\; for\; j = 1, \ldots ,d\\                                                                                                        
 K(0, \xi) = 1 &\;\;for\; \xi\in\mathbb{R}^d.                                                                                                       \end{array}\right.$$
An other description was discovered by R. Howe \cite{How}. It states that $\mathcal{F}$ can also be defined by
\begin{equation}\label{def2}
\mathcal{F} = exp\left(\frac{i\pi d}{4}\right)\,exp\left(\frac{i\pi}{4}(\Delta-\|x\|^2)\right),
\end{equation}
where $\Delta$ is the Laplacian operator on $\mathbb{R}^d$. S. Ben Said, T. Kobayashi and B.Orsted \cite{BKO1}
gave an extension of (\ref{def2}) by replacing the Euclidean Laplace operator $\Delta$ by the sum of squares
$\Delta_\kappa$ of Dunkl operators associated with a given finite reflection group in $\mathbb{R}^d$. They defined a generalized Fourier transform $\mathcal{F}_{\kappa,a}$, where $a$ is a positive real parameter  coming from the interpolation of the minimal
unitary representations of two different reductive Lie groups and $\kappa$ is a real parameter coming from Dunkl's theory (see \cite{Du}). Many other researchs deal with this operator. One can cite for example \cite{BKO,HD,GIT}.

In this paper, we consider the case $a=\frac{2}{n}$, where
 $n\in \mathbb{N}^\star$ and $\kappa >\displaystyle \frac{n-1}{2n}$.
\\ Lets define the measure
\begin{equation}\label{mu}
d\mu_{\kappa,n}(x)= \left(M_{\kappa,n}\right)^{-1} |x|^{2\kappa+\frac{2}{n}-2}dx,\end{equation}
where
\begin{equation}\label{constant}
M_{\kappa,n}=2\left(\frac{2}{n}\right)^{\kappa n-\frac{n}{2}}\Gamma\left(\kappa n+ 1 -\frac{n}{2}\right).
\end{equation}
We denote $\mathcal{F}_{\kappa,n}=\mathcal{F}_{\kappa,\frac{2}{n}}$, the generalized Hankel transform, defined for $f\in L^1(\mathbb{R},d\mu_{\kappa,n})$, by
\begin{equation}\label{Fka}
 \mathcal{F}_{\kappa,n}f(\lambda)=\displaystyle\int_{\mathbb{R}}f(x)B_\lambda^{\kappa,n}(x)d\mu_{\kappa,n}(x),\quad\lambda \in \mathbb{R},
\end{equation}
where
  $B_\lambda^{\kappa,n}$ is the generalized Hankel function, defined by
	\begin{equation}\label{1.1} B_\lambda^{{\kappa,n}}(x)=j_{\kappa n-\frac{n}{2}}\left(n|\lambda x|^\frac{1}{n}\right)+(-i)^n\left(\frac{n}{2}\right)^n
\frac{\Gamma(\kappa n-\frac{n}{2}+1)}{\Gamma(\kappa n+\frac{n}{2}+1)}\lambda xj_{\kappa n+\frac{n}{2}}\left(n|\lambda x|^\frac{1}{n}\right),\end{equation}
where $j_\alpha$ is the normalized Bessel function of first kind and order $\alpha$, given by
	\begin{equation}\label{2.5}
	j_\alpha(z)=\Gamma(\alpha+1)\left(\frac{z}{2}\right)^{-\alpha}J_\alpha(z).
	\end{equation}
Note that $J_\alpha$ is the Bessel function of first kind and index $\alpha>-\frac{1}{2}$.

Our aim in this paper is to establish the following product formula
\begin{equation}\label{kernel}
B^{\kappa,n}_\lambda(x)\,B^{\kappa,n}_\lambda(y) = \int_\mathbb{R} B^{\kappa,n}_\lambda(z) \,d\nu_{x,y}^{\kappa,n}(z),\quad \lambda,\:x,\:y\in
 \mathbb{R}.
\end{equation}
Here the measure $d\nu_{x,y}^{\kappa,n}$ is finded explicitly by relation (\ref{main}). It is real valued measure, compactly supported on $\mathbb{R}$. It may not be positive and it is uniformly bounded on $x, y \in \mathbb{R}$.
\\ Formula (\ref{kernel}) recover the cases $n=1$ for Dunkl kernel on $\mathbb{R}$ in \cite{Ro} and $n=2$ for a modified Hankel kernel in \cite{Sa}.

Analougous results are obtained by Flensted-Jensen and Koornwinder \cite{FK} for Jacobi functions, and of Ben Salem
and Ould Ahmed Salem \cite{BO} for Jacobi-Dunkl functions, and Anker, Ayadi and the third author \cite{AAS} for Opdam's hypergeometric functions.

The key tool in the proof of (\ref{kernel}) is the relation (\ref{nqq}) of Theorem \ref{th0} which gives an integral representation  of the product $u^nj_{\alpha+n}(u)j_{\alpha}(v)$, for  $u,v\in [0,+\infty[$,
$n\in \mathbb{N}^\star$ and $\alpha>-1/2$, invoking Gegenbauer polynomials. This last formula seems to be new. %T. Koornwinder, in a personal communication, suggest to call it inductive product formula for modified Bessel functions.

The product formula (\ref{kernel}) allows us to define and
study the translation operators
$$\tau_x^{\kappa,n}f(y)=\int_{\mathbb{R}}f(z)d\nu_{x,y}^{\kappa,n}(z).$$ In particular, we find that $\tau_x^{\kappa,n}$ is a bounded operator on $L^p(\mathbb{R},d\mu_{\kappa,n})$.
Next, we introduce the convolution product for suitable functions $f,g$, by
$$f\star_{\kappa,n}g(x)=\int_{\mathbb{R}}f(y)\tau_x^{\kappa,n}g(y)d\mu_{\kappa,n}(z).$$
We show in particular that $f\star_{\kappa,n}g=g\star_{\kappa,n}f$ and $\mathcal{F}_{\kappa,n}(f\star_{\kappa,n}g)=
\mathcal{F}_{\kappa,n}(f)\mathcal{F}_{\kappa,n}(g)$.\\

We now briefly summarize the content of the remaining sections of the paper. In section 2 we recall some important properties of Bessel functions and Gegenbauer polynomials which will be used later. Section 3 is devoted to establish an integral form of the product $u^nj_{\alpha+n}(u)j_{\alpha}(v)$.  Section 4
deals with our main result: there we prove the product formula (\ref{kernel}).Furthermore, we close this section by giving some properties of the measure $d\nu_{x,y}^{\kappa,n}$.  In Section 5, we make more detailed study of the translation operator $\tau_x^{\kappa,n}$ and the
associated convolution product $\star_{\kappa,n}$.
\section{Preliminaries}
In this section we recall some properties of the Gegenbauer polynomials  and of the Bessel functions. See (\cite{AnAsRo,As,Wa}) for more details.

For $\alpha>0$ and $m \in\mathbb{N}$, the Gegenbauer polynomials $C_m^\alpha$, are defined by
$$C_m^\alpha(t)= \frac{\Gamma(m+2\alpha)}{\Gamma(2\alpha)\Gamma(m+1)}\;_2F_1\left[\begin{array}{c}-m\;\;m+2\alpha\\\alpha+\frac{1}{2}\end{array};
\,\frac{1-t}{2}\right], $$ where $_2F_1$ is the Gauss hypergeometric function.
Explicitly, it takes the form
\begin{equation}\label{Gegenbauer}
C_m^\alpha(t)=\frac{1}{\Gamma(\alpha)}\sum_{k=0}^{[m/2]}(-1)^k\frac{\Gamma(m-k+\alpha)}{k!(m-2k)!}(2t)^{m-2k}.
\end{equation}
Furthermore, we have
\begin{equation}\label{boudednessC}
|C^{\alpha}_{m}(\cos\phi)|\leq C^{\alpha}_{m}(1).
\end{equation}
Gegenbauer polynomials could be defined by the recurrence formula:
\begin{eqnarray}\nonumber
C_0^\alpha(t)&=&1,\\ \nonumber
C_1^\alpha(t)&=&2\alpha t,\\
m C_m^\alpha(t)&=& 2(m+\alpha-1)t C_{m-1}^\alpha(t)-(m+2\alpha-2)C_{m-2}^\alpha(t),\quad m\geq 2.\label{Gegnrec}
\end{eqnarray}
In particular, we get
$$C_2^\alpha(t)= 2\alpha(\alpha+1)t^2-\alpha ,$$
\begin{equation}\label{C3}C_3^\alpha(t)= \frac{2^2}{3}\alpha(\alpha+1)(\alpha+2)t^3-2\alpha(\alpha+1)t,
\end{equation}
and \begin{equation}\label{C4}C_4^\alpha(t)= \frac{4}{6}\alpha(\alpha+1)(\alpha+2)(\alpha+3)t^4-12\alpha(\alpha+1)(\alpha+2)t^2+3\alpha(\alpha+1).
\end{equation}
Gegenbauer polynomials can also be defined by the generating function
$$\label{generatic}
  (1-2tr+r^2)^{-2\alpha}=\sum_{m=0}^\infty C_m^\alpha(t)r^m,
$$
from which it follows that
\begin{equation}\label{gegen}
  C_m^\alpha(1)=\displaystyle \frac{(2\alpha)_{m}}{m!};\quad C_m^\alpha(-1)=(-1)^m\displaystyle \frac{(2\alpha)_{m}}{m!},
\end{equation}
where $(\alpha)_{m}=\displaystyle\frac{\Gamma(\alpha+m)}{\Gamma(\alpha)}$ is the Pochhammer symbol.

Gegenbauer polynomials form an orthogonal basis in $L^2((-1,1),(1-t^2)^{\alpha-1/2}dt)$ and we have
\begin{equation}\label{2.4}
\int_{-1}^1 C_m^\alpha(t) C_n^\alpha(t)(1-t^2)^{\alpha-1/2}\, dt=\frac{\pi\Gamma(2\alpha+m)}{2^{2\alpha-1}\Gamma(m+1)(m+\alpha)\Gamma(\alpha)^2} \delta_{nm}.
\end{equation}
The Bessel function $J_\alpha$ of first kind  and index $\alpha>-1/2$, is given by
$$J_\alpha(x)=\sum_{n=0}^{\infty}\frac{(-1)^n}{n!\,\Gamma(\alpha+n+1)}\left(\frac{x}{2}\right)^{\alpha+2n}.$$
There is a relationship between Gegenbauer polynomials and the Bessel functions which is known as Gegenbauer's addition theorem and it states, for $\alpha=0$,
\begin{equation}\label{2.1}
    J_0(\sqrt{u^2+v^2-2uv\cos \phi})=J_0(u)J_0(v)+2\sum_{m=1}^\infty J_m(u)J_m(v)\cos m\phi, \quad 0 \leq \phi \leq \pi,
\end{equation}
and for general $\alpha$,
\begin{equation}\label{2.2}
\frac{J_\alpha(\sqrt{u^2+v^2-2uv\cos(\phi)})}{\{\sqrt{u^2+v^2-2uv\cos(\phi)}\}^\alpha}= 2^\alpha\Gamma(\alpha)\sum_{m=0}^\infty(\alpha+m)
\frac{J_{\alpha+m}(u)}{u^\alpha}\frac{J_{\alpha+m}(v)}{v^\alpha} C_m^\alpha(\cos \phi).
\end{equation}
A consequence of (\ref{2.4}) and Gegenbauer's addition formulas (\ref{2.1}) and (\ref{2.2}) is the Sonine product formula given, for $\alpha>-\frac{1}{2}$, by
$$\label{2.6}
j_\alpha(u) j_\alpha(v)=C_{\alpha}\int_0^\pi j_\alpha \{\sqrt{u^2+v^2-2uv\cos \phi}\}(\sin\phi)^{2\alpha}d\phi,
$$
where $C_{\alpha}=\displaystyle \frac{\Gamma(\alpha+1)}{\sqrt{\pi}\Gamma(\alpha+1/2)}$. \\
Making the substitution
$$ w=\sqrt{u^2+v^2-2uv\cos \phi}, \quad 0\leq \phi\leq\pi,$$ we get the Bessel product formula
\begin{equation}\label{2.7}
j_\alpha(\lambda u) j_\alpha(\lambda v)=\int_0^\infty j_\alpha(\lambda w) K_B^\alpha(u,v,w)\, w^{2\alpha+1} dw,
\end{equation}
where \begin{equation}\label{2.8}
K_B^\alpha(u,v,w)=2^{-2\alpha+1}C_{\alpha}\frac{\{[(u+v)^2-w^2][w^2-(u-v)^2]\}^{\alpha-\frac{1}{2}}}{(uvw)^{2\alpha}}
\end{equation}
for $|u-v| \leq w\leq u +v$ and $K^\alpha_B(u, v, w) =0$ elsewhere.
\\ $K^\alpha_B$ is homogeneous of degree $-2\alpha-2,$ i.e.
\begin{equation}\label{2.9}
K^{\alpha}_{B}(\lambda u,\lambda v,\lambda w)=\lambda^{-2\alpha-2}K^{\alpha}_{B}(u,v,w), \ \ \lambda>0.
\end{equation}
Recall that \begin{equation}\label{Kalpha}
\int_0^\infty K_B^\alpha(u,v,w)\, w^{2\alpha+1} dw=1.
\end{equation}

We close this section by recalling the Sonine's integral formula
\begin{equation}\label{2.16}
  j_\alpha(x)=\frac{2\Gamma(\alpha+1)}{\Gamma(\beta+1)\Gamma(\alpha-\beta+1)}\int_0^1j_\beta(xt)(1-t^2)^{\alpha-\beta-1}t^{2\beta+1}dt,\quad \alpha>\beta>-\frac{1}{2},
\end{equation}
the derivative formula
\begin{equation}\label{Dj}
 \frac{d}{dx}j_\alpha(x)=\frac{-1}{2(\alpha+1)}xj_{\alpha+1}(x)
\end{equation}
and the three-term recurrence relation
\begin{equation}\label{rec}
j_{\alpha+n-1}(u)=j_{\alpha+n-2}(u)+\frac{u^2}{4(\alpha+n-1)(\alpha+n)}j_{\alpha+n}(u),\quad n\in \mathbb{N}^\star.
\end{equation}
\section{An inductive product formula for modified Bessel functions}
Our main result in this section concerns an integral form of the product $u^nj_{\alpha+n}(u)j_{\alpha}(v)$,
 $u,v\in ]0,+\infty[$ and $n\in\mathbb{N}$.

We shall use the following abbreviation:
$$\label{3.1}
 \delta(u,v,\phi)=\sqrt{u^2+v^2-2uv \cos\phi},\quad u,v>0;\quad \phi \in ]0,\pi[.
$$
 Recall that (see \cite[Subsection 2.2]{Sa})
\begin{eqnarray}\nonumber
  u j_{\alpha+1}(u)j_{\alpha}(v)&=&C_\alpha\,\int_{0}^{\pi}
j_{\alpha+1}\left( \delta(u,v,\phi)\right)\,(u-v\cos\phi) (\sin\phi)^{2\alpha}d\phi\\ \nonumber
&=&\frac{C_\alpha}{2\alpha}\int_{0}^{\pi}j_{\alpha+1}\left( \delta(u,v,\phi)\right)\left( \delta(u,v,\phi)\right)
  C_1^\alpha\left(\frac{u-v\cos\phi}{\delta(u,v,\phi)}\right)(\sin\phi)^{2\alpha}d\phi,\\ \label{n=1}
\end{eqnarray}
and
$\displaystyle u^2 j_{\alpha+2}(u)j_{\alpha}(v)$
\begin{eqnarray}\nonumber
 & =&\frac{C_\alpha}{2\alpha+1}\int_{0}^{\pi}j_{\alpha+2}\left( \delta(u,v,\phi)\right)
 \, \left\{2(\alpha+1)\left(u-v\cos\phi\right)^2
   -\left(\delta(u,v,\phi)\right)^2\right\}(\sin\phi)^{2\alpha}d\phi\\
   &=&\frac{2C_\alpha}{(2\alpha)_2}\int_{0}^{\pi}j_{\alpha+2}\left( \delta(u,v,\phi)\right)\left( \delta(u,v,\phi)\right)^2
  C_2^\alpha\left(\frac{u-v\cos\phi}{\delta(u,v,\phi)}\right)(\sin\phi)^{2\alpha}d\phi.\label{n=2}
\end{eqnarray}

 To perform the recurrence product formula, we start by considering the case $n=3$. \\

Differentiating both sides of (\ref{n=2}) with respect to $u$ and applying the derivative formula (\ref{Dj}) and the
three-term recurrence relation (\ref{rec}), yield to\\

$\displaystyle 2u j_{\alpha+1}(u)j_{\alpha}(v)- \frac{(\alpha+1)u^3}{2(\alpha+2)(\alpha+3)}j_{\alpha+3}(u)j_{\alpha}(v)$
\begin{eqnarray*}&=& 2C_\alpha\,
\int_{0}^{\pi}(u-v\cos\phi)j_{\alpha+1}\left( \delta(u,v,\phi)\right) (\sin\phi)^{2\alpha}d\phi+ \frac{C_\alpha}{2\alpha+1}\int_{0}^{\pi}
j_{\alpha+3}\left( \delta(u,v,\phi)\right)(\sin\phi)^{2\alpha}\\
  &\times& \left\{\frac{3(\alpha+1)}{2(\alpha+2)(\alpha+3)}(u-v\cos\phi)\left(\delta(u,v,\phi)\right)^2- \frac{\alpha+1}{\alpha+3}
    \left(u-v\cos\phi\right)^3\right\}d\phi.
\end{eqnarray*}
Invoking (\ref{n=1}) and (\ref{C3}), we get
\begin{eqnarray}\nonumber
  \displaystyle u^3 j_{\alpha+3}(u)j_{\alpha}(v)& =&\frac{C_\alpha}{2\alpha+1}\int_{0}^{\pi}(\sin\phi)^{2\alpha}
  j_{\alpha+3}\left( \delta(u,v,\phi)\right)\\ &\times& \nonumber
  \left\{2(\alpha+2)\left(u-v\cos\phi\right)^3 -3(u-v\cos\phi)\left(\delta(u,v,\phi)\right)^2\right\}d\phi.\\ \nonumber
  &=&\frac{6C_\alpha}{(2\alpha)_3}\int_{0}^{\pi}j_{\alpha+3}\left( \delta(u,v,\phi)\right)\left( \delta(u,v,\phi)\right)^3
  C_3^\alpha\left(\frac{u-v\cos\phi}{\delta(u,v,\phi)}\right)(\sin\phi)^{2\alpha}d\phi.\\ \label{n=3}
\end{eqnarray}

Secondly, by differentiating both sides of (\ref{n=3}) it follows from the three-term recurrence relation\\

$\displaystyle 3u^2 j_{\alpha+2}(u)j_{\alpha}(v)- \frac{(2\alpha+3)u^4}{4(\alpha+3)(\alpha+4)}j_{\alpha+4}(u)j_{\alpha}(v)$
\begin{eqnarray*}&=& \frac{3C_\alpha}{2\alpha+1}\,
 \int_{0}^{\pi}(\sin\phi)^{2\alpha}j_{\alpha+2}\left( \delta(u,v,\phi)\right)\left\{2(\alpha+1)(u-v\cos\phi)^2-\left(\delta(u,v,\phi)\right)^2\right\} d\phi\\
&-& \frac{C_\alpha}{2\alpha+1}\int_{0}^{\pi}(\sin\phi)^{2\alpha}j_{\alpha+4}\left( \delta(u,v,\phi)\right)\left\{\frac{\alpha+2}{\alpha+4}(u-v\cos\phi)^4\right.\\
  &-& \left. 3\frac{\alpha+2}{(\alpha+4)(\alpha+4)}
    \left(u-v\cos\phi\right)^2\left(\delta(u,v,\phi)\right)^2+
  \frac{3}{4(\alpha+3)(\alpha+4)}\left(\delta(u,v,\phi)\right)^4\right\}d\phi.
\end{eqnarray*}
In similar way as above, having in mind (\ref{n=2}) and (\ref{C4}),  we obtain
\begin{eqnarray}\nonumber
  u^4 j_{\alpha+4}(u)j_{\alpha}(v)&=&\frac{C_\alpha}{(2\alpha+1)(2\alpha+3)}\int_{0}^{\pi}
  j_{\alpha+4}\left( \delta(u,v,\phi\right)\,\left\{4(\alpha+2)(\alpha+3)\left(u-v\cos\phi\right)^4\right.\\ \nonumber
  &-& \left. 12(\alpha+2)(u-v\cos\phi)^2\left(\delta(u,v,\phi)\right)^2
 + 3\left(\delta(u,v,\phi)\right)^4\right\}(\sin\phi)^{2\alpha}d\phi.\\ \nonumber
  &=&\frac{24C_\alpha}{(2\alpha)_4}\int_{0}^{\pi}j_{\alpha+4}\left( \delta(u,v,\phi)\right)\left( \delta(u,v,\phi)\right)^4
 C_4^\alpha\left(\frac{u-v\cos\phi}{\delta(u,v,\phi)}\right)(\sin\phi)^{2\alpha}d\phi.\\ \label{n=4}
\end{eqnarray}

Taking into account (\ref{n=1}), (\ref{n=2}), (\ref{n=3}) combined with (\ref{n=4}), we can formulate the main result of this section.
\begin{theorem}\label{th0} Let $u,v\in [0,+\infty[$. Then for all $n\in \mathbb{N}$, we have\\

$\displaystyle u^{n} j_{\alpha+n}(u)j_{\alpha}(v)=C_{\alpha}\displaystyle \frac{n!}{(2\alpha)_n}\displaystyle\int_{0}^{\pi}j_{\alpha+n}\left( \sqrt{u^2+v^2-2uv \cos\phi}\right)$
\begin{equation}\label{recurr} \times\left\{
\left( u^2+v^2-2uv \cos\phi\right)^{\frac{n}{2}}
C_n^\alpha\left(\frac{u-v\cos\phi}{\sqrt{u^2+v^2-2uv \cos\phi}}\right)(\sin\phi)^{2\alpha}d\phi\right\}.
\end{equation}
 Or equivalently
\begin{equation}
u^{n} j_{\alpha+n}(u)j_{\alpha}(v)=\displaystyle \frac{n!}{(2\alpha)_n}\displaystyle\int_{0}^{+\infty}w^nj_{\alpha+n}(w) K_B^\alpha(u,v,w)
C^{\alpha}_{n}\left(\frac{u^2+w^2-u^2}{2uw}\right)w^{2\alpha+1}dw.\label{nqq}
\end{equation}
\end{theorem}

A key tool in the proof of Theorem \ref{th0} is the following lemma.\\
We shall use the following abbreviation. For $n\in \mathbb{N}^\star$, we denote by $\psi_n$, the function defined by
\begin{equation}\label{3.7}
 \psi_n(u,v,\phi)=\displaystyle \frac{n!}{(2\alpha)_n}\left( u^2+v^2-2uv \cos\phi\right)^{\frac{n}{2}}
C_n^\alpha\left(\frac{u-v\cos\phi}{\sqrt{u^2+v^2-2uv \cos\phi}}\right).
\end{equation}
\begin{lemma}
  For $n\in \mathbb{N}^\star$, the derivative of the function $\psi_n$ with respect to the variable $u$ satisfies
\begin{equation}\label{key}
\psi'_n(u,v,\phi)=n
\psi_{n-1}(u,v,\phi).
\end{equation}
\begin{proof}
  It follows from (\ref{Gegenbauer}) that
$$\psi_n(u,v,\phi)=\frac{n!}{\Gamma(\alpha)(2\alpha)_n}
\sum_{k=0}^{[n/2]}(-1)^k\frac{2^{n-2k}\Gamma(n-k+\alpha)}{k!(n-2k)!}\left(u-v\cos\phi\right)^{n-2k}
\left(u^2+v^2-2uv \cos\phi\right)^{k}. $$
\\ {\bf Case 1: $n$ is an even integer.} Clearly\\

$\displaystyle \psi'_n(u,v,\phi)$
\begin{eqnarray*} &=&
\frac{n!}{\Gamma(\alpha)(2\alpha)_n}\left\{\sum_{k=0}^{\frac{n}{2}-1}(-1)^k
\frac{2^{n-2k}\Gamma(n-k+\alpha)}{k!(n-2k-1)!}\left(u-v\cos\phi\right)^{n-2k-1} \left(u^2+v^2-2uv \cos\phi\right)^{k}\right.\\
&+&\left.\sum_{k=1}^{\frac{n}{2}}(-1)^k\frac{2^{n-2k+1}\Gamma(n-k+\alpha)}{(k-1)!(n-2k)!}\left(u-v\cos\phi\right)^{n-2k+1}
 \left(u^2+v^2-2uv \cos\phi\right)^{k-1}\right\}.
\end{eqnarray*}
Therefore since $n$ is an even integer which gives $\displaystyle \frac{n}{2}=[\frac{n-1}{2}]+1$, and so\\

$\displaystyle \psi'_n(u,v,\phi)$
\begin{eqnarray*} &=&\displaystyle \frac{n!}{\Gamma(\alpha)(2\alpha)_{n-1}}
\sum_{k=0}^{[\frac{n-1}{2}]}(-1)^k\frac{2^{n-2k}\Gamma(n-k+\alpha)}{k!(n-2k-1)!}
\left(u-v\cos\phi\right)^{n-1-2k} \left(u^2+v^2-2uv \cos\phi\right)^{k}\\
&=& n\psi_{n-1}(u,v,\phi).
\end{eqnarray*}
{\bf Case 2: $n$ is an odd integer.} Observe that\\

$\displaystyle \psi'_n(u,v,\phi)$
\begin{eqnarray*} &=&\frac{n!}{\Gamma(\alpha)(2\alpha)_n}\left\{
\sum_{k=0}^{[\frac{n}{2}]}(-1)^k\frac{2^{n-2k}\Gamma(n-k+\alpha)}{k!(n-2k-1)!}\left(u-v\cos\phi\right)^{n-2k-1}
\left(u^2+v^2-2uv \cos\phi\right)^{k}\right.\\
& &\left. +\sum_{k=1}^{[n/2]}(-1)^k\frac{2^{n-2k+1}\Gamma(n-k+\alpha)}{(k-1)!(n-2k)!}\left(u-v\cos\phi\right)^{n-2k+1}
  \left(u^2+v^2-2uv \cos\phi\right)^{k-1}\right\}.
\end{eqnarray*}
Since $n$ is an odd integer, then $\displaystyle [\frac{n}{2}]=[\frac{n-1}{2}]$, which implies\\

$\displaystyle \psi'_n(u,v,\phi)$
\begin{eqnarray*} &=&
\frac{n!}{\Gamma(\alpha)(2\alpha)_{n-1}}\sum_{k=0}^{[\frac{n}{2}]}(-1)^k
\frac{2^{n-2k}\Gamma(n-k+\alpha)}{k!(n-2k-1)!}\left(u-v\cos\phi\right)^{n-2k-1} \left(u^2+v^2-2uv \cos\phi\right)^{k}\\ &=&
n\psi_{n-1}(u,v,\phi).
\end{eqnarray*}
Thus (\ref{key}) holds for all integer $n\geq1$.
\end{proof}
\end{lemma}
\noindent\textbf{Proof of Theorem \ref{th0}.}
We proceed by induction. Suppose that the statement (\ref{recurr}) holds for any given $n$.\\
Differentiate (\ref{recurr}) with respect to the variable $u$, then in view of the three-term recurrence relation (\ref{rec}) and (\ref{key}), we obtain\\

$\displaystyle nu^{n-1}j_{\alpha+n-1}(u)j_{\alpha}(v)-\frac{2\alpha+n}{4(\alpha+n)(\alpha+n+1)}u^{n+1}j_{\alpha+n+1}(u)j_{\alpha}(v)$
\begin{eqnarray}\nonumber
  &=&\nonumber  C_\alpha\,\int_0^\pi n\psi_{n-1}(u,v,\phi)j_{\alpha+n-1}\left(\sqrt{u^2+v^2-2uv \cos\phi}\right)
  (\sin\phi)^{2\alpha}d\phi \\ \nonumber
 & -& \frac{1}{2(\alpha+n+1)}C_\alpha\,\int_0^\pi (\sin\phi)^{2\alpha}\, j_{\alpha+n+1}\left(\sqrt{u^2+v^2-2uv \cos\phi}\right)\\
 &\times& \displaystyle \left\{(u-v\cos\phi)\psi_{n}(u,v,\phi)- \frac{n}{2(\alpha+n)}\left(u^2+v^2-2uv \cos\phi\right)
   \psi_{n-1}(u,v,\phi)\right\}
  d\phi.\label{3.8}
\end{eqnarray}
Using the recurrence hypothesis (\ref{3.8}) simplifies to give\\

$\displaystyle u^{n+1}j_{\alpha+n+1}(u)j_{\alpha}(v)$
 \begin{eqnarray*}  \nonumber & =& \frac{C_\alpha}{2\alpha+n}\int_0^\pi (\sin\phi)^{2\alpha} \,j_{\alpha+n+1}\left(\sqrt{u^2+v^2-2uv \cos\phi}\right)
 \\ &\times& \displaystyle \left\{2(\alpha+n)(u-v\cos\phi)\psi_{n}(u,v,\phi)- n\psi_{n-1}(u,v,\phi)\left(u^2+v^2-2uv \cos\phi\right)
   \right\}
  d\phi.\label{3.9}\end{eqnarray*}
  Having in mind (\ref{3.7}) we obtain by  direct calculations \\

  $\displaystyle 2(\alpha+n)(u-v\cos\phi)\psi_{n}(u,v,\phi)- n\psi_{n-1}(u,v,\phi)\left(u^2+v^2-2uv \cos\phi\right)$
  \begin{eqnarray*}
     &=& \frac{n!}{(2\alpha)_n}
    \left\{2(\alpha+n)\frac{u-v\cos\phi}{\sqrt{u^2+v^2-2uv \cos\phi}} C_n^\alpha
    \left(\frac{u-v\cos\phi}{\sqrt{u^2+v^2-2uv \cos\phi}}\right)
     \right.\\ &-& \left.
     (2\alpha+n-1)C_{n-1}^\alpha\left(\frac{u-v\cos\phi}{\sqrt{u^2+v^2-2uv \cos\phi}}\right)\right\}\left(u^2+v^2-2uv \cos\phi\right)^{\frac{n+1}{2}}.
  \end{eqnarray*}
From the recurrence formula of Gegenbauer polynomials (\ref{Gegnrec}), we rewrite the equality (\ref{3.9}) in the following form
\begin{eqnarray*}u^{n+1}j_{\alpha+n+1}(u)j_{\alpha}(v)&=&\displaystyle \frac{C_\alpha(n+1)!}{(2\alpha)_{n+1}}\int_{0}^\pi
C_{n+1}^\alpha\left(\frac{u-v\cos\phi}{\sqrt{u^2+v^2-2uv \cos\phi}}\right)\\
&\times &\left(u^2+v^2-2uv \cos\phi\right)^{\frac{n+1}{2}}
j_{\alpha+n+1}\left(\sqrt{u^2+v^2-2uv \cos\phi}\right)(\sin\phi)^{2\alpha}d\phi.
\end{eqnarray*}
From which we obtain the recursion relation (\ref{recurr}).\\
Making the substitution $w=\displaystyle \sqrt{u^2+v^2-2uv \cos\phi}$ with the help of the following expression
$$\frac{u-v\cos\phi}{\sqrt{u^2+v^2-2uv \cos\phi}}=\frac{u^2+w^2-v^2}{2uw},$$ we find (\ref{nqq}) which completes the proof of Theorem \ref{th0}.
$\hfill\Box$
\section{The product formula for generalized Hankel functions}
We will now derive a product formula with a  non positive kernel for the generalized Hankel functions 	$B_\lambda^{\kappa,n}$
on $\mathbb{R}$. We shall use the following abbreviations:\\
\noindent\textbf{Notations.} For $x,y\in \mathbb{R}^\star$, $z\in \mathbb{R}$, $n\in \mathbb{N}^\star$ and $\kappa>\frac{n-1}{2n}$, let
\begin{equation}\label{Delta}
\sigma_{x,y,z}^n=\frac{|x|^{\frac{2}{n}}+|y|^{\frac{2}{n}}-|z|^{\frac{2}{n}}}{2| xy|^{\frac{1}{n}}},
\end{equation}
as well as
\begin{equation}\label{ksi}\xi_{\kappa,n}(x,y,z)=\frac{n!\textrm{sgn}(xy)}{(2\kappa n-n)_n}C_n^{\kappa n-\frac{n}{2}}\left(\sigma_{x,y,z}^n\right).\end{equation}
Here, $\textrm{sgn} (x)$ stands for the sign of a real number $x$.\\
Furthermore let us define
\begin{eqnarray}\nonumber
%  to remove numbering (before each equation)
  \mathcal{K}_{\kappa,n}(x,y,z)& = &\frac{M_{\kappa,n}}{2n} K_B^{\kappa n-\frac{n}{2}}\left(|x|^\frac{1}{n},|y|^\frac{1}{n},|z|^\frac{1}{n}\right)\\ &\times&\nonumber \left\{1+ (-1)^n \xi_{\kappa,n}(x,y,z)+\xi_{\kappa,n}(z,x,y)+\xi_{\kappa,n}(y,z,x)\right\}.\\ \label{calK}
\end{eqnarray}
 Here  $K^{\alpha}_{B}$, $C_n^\alpha$ and $M_{\kappa,n}$ are as in (\ref{2.8}), (\ref{Gegenbauer}) and (\ref{constant}).

The following obvious  properties  will play an important role in what
follows:
\noindent\begin{properties} Let $n\in \mathbb{N}^\star$ and $\kappa>\displaystyle \frac{n-1}{2n}$, then
\begin{itemize}
\item [i)]  The mapping $(x,y,z)\rightarrow \sigma^n_{x,y,z}$ is homogeneous of degree $0$.
\item [ii)] We have
\begin{equation}\label{maj}
\left|\xi_{\kappa,n}(x,y,z)\right|\leq1, \quad x,y,z\in\mathbb{R}^\star.
\end{equation}
\item  [iii)] The kernel $\mathcal{K}_{\kappa,n}$ satisfies
\end{itemize}
$$\left\{\begin{array}{ccc}
  \mathcal{K}_{\kappa,n}(x,y,z) &=& \mathcal{K}_{\kappa,n}(y,x,z) \\
 \nonumber \mathcal{K}_{\kappa,n}(x,y,z) &=& \mathcal{K}_{\kappa,n}((-1)^n x,z,y) \\
  \mathcal{K}_{\kappa,n}(x,y,z)&=& \mathcal{K}_{\kappa,n}(z,(-1)^ny,x).  \end{array}\right.
$$
\end{properties}
\begin{proof}
  i) is clear.\\
  ii) follows from (\ref{boudednessC}).\\
  iii) is immediate.
\end{proof}

 We now state and prove the main result in the paper.
\begin{theorem}\label{th1} For $\lambda,\:x,\:y\in \mathbb{R}$, we have
$$
B^{\kappa,n}_\lambda(x)\,B^{\kappa,n}_\lambda(y) = \int_\mathbb{R} B^{\kappa,n}_\lambda(z) \,d\nu_{x,y}^{\kappa,n}(z),
$$
where
\begin{equation}\label{main} d\nu_{x,y}^{\kappa,n}(z)=\left\{\begin{array}{ccc}
                                    \mathcal{K}_{\kappa,n}(x,y,z)d\mu_{\kappa,n}(z) & \textrm{if} & xy\neq 0 \\
                                    d\delta_x(z) & \textrm{if} & y=0 \\
                                    d\delta_y(z) & \textrm{if} & x=0
                                  \end{array}\right.
\end{equation}
and $d\mu_{\kappa,n}$ is as in (\ref{mu}).\end{theorem}

Before proving the theorem, we need the following lemma.
\begin{lemma} Let $u,v\in [0,+\infty[$. Then for all $n\in \mathbb{N}$, we have\\

  $\displaystyle \frac{(uv)^n}{2^{2n}}j_{\alpha+n}(u)j_{\alpha+n}(v)$
\begin{equation}\label{key1}
 =\frac{n!\left((\alpha+1)_n\right)^2}{(2\alpha)_n}\int_0^\infty j_\alpha(w)K_B^\alpha(u,v,w)\,C_n^\alpha\left(\frac{u^2+v^2-w^2}{2uv}\right)w^{2\alpha+1} dw.\end{equation}
\end{lemma}
\begin{proof} Gegenbauer's addition formula (\ref{2.2}) for Bessel functions leads to \\

$\displaystyle \int_0^\pi \frac{J_\alpha(\sqrt{u^2+v^2-2uv\cos \phi})}{\{\sqrt{u^2+v^2-2uv\cos \phi}\}^\alpha}C_n^\alpha(\cos \phi)(\sin\phi)^{2\alpha}d\phi$
$$  =2^\alpha \Gamma(\alpha)(\alpha+n)\frac{J_{\alpha+n}(u)}{u^\alpha}\frac{J_{\alpha+n}(v)}{v^\alpha}\int_0^\pi\left[C_n^\alpha(\cos \phi)\right]^2(\sin\phi)^{2\alpha}d\phi.
$$
By virtue of  (\ref{2.4}) and (\ref{2.5}), this leads to \\

$\displaystyle\frac{(uv)^n}{2^{2n}}j_{\alpha+n}(u)j_{\alpha+n}(v)$
$$
=\frac{n!\Gamma(\alpha)\Gamma(\alpha+n+1)^2}{\pi 2^{-2\alpha+1}\Gamma(2\alpha+n)\Gamma(\alpha+1)}\int_0^\pi j_\alpha(\sqrt{u^2+v^2-2uv\cos \phi})C_n^\alpha(\cos \phi)(\sin\phi)^{2\alpha}d\phi.$$
Substitution of $w:=\sqrt{u^2+v^2-2uv\cos \phi}$ and the duplication formula yield the
statement.
\end{proof}
\noindent\textit{Proof of Theorem \ref{th1}.} Recall the generalized Hankel functions $B^{\kappa,n}_\lambda$
 defined in (\ref{1.1}).\\
Let us split the generalized Hankel kernel in two parts
$$B^{\kappa,n}_\lambda=B^{\kappa,n}_{\lambda,e}+B^{\kappa,n}_{\lambda,o}$$
into its even part
$$B^{\kappa,n}_{\lambda,e}(x)=j_{\kappa n-\frac{n}{2}}\left(n|\lambda x|^\frac{1}{n}\right)$$
and its odd part
$$B^{\kappa,n}_{\lambda,o}(x)= (\frac{n}{2})^n\frac{(-i)^n\lambda x}{(\kappa n-\frac{n}{2}+1)_n}j_{\kappa n+\frac{n}{2}}\left(n|\lambda x|^\frac{1}{n}\right).$$
 Setting $\alpha=\kappa n-\frac{n}{2}$, $u=n|\lambda x|^\frac{1}{n}$, $v=n| \lambda y|^\frac{1}{n}$, and rewrite the integrand with the substitution
\begin{equation}\label{change}
  w=n(|\lambda| z)^\frac{1}{n},\quad z\in \mathbb{R}_+.
\end{equation}
 (\ref{2.7}) becomes\\

 $\displaystyle B^{\kappa,n}_{\lambda,e}(x)\,B^{\kappa,n}_{\lambda,e}(y)$
 $$
=n^{2\kappa n-n+1}|\lambda|^{2\kappa+\frac{2}{n}-1}\int_0^\infty B^{\kappa,n}_{\lambda,e}(z)
K_B^{\kappa n-\frac{n}{2}}(n|\lambda x|^{\frac{1}{n}},n|\lambda y|^{\frac{1}{n}},n|\lambda|z^{\frac{1}{n}})\,z^{2\kappa+\frac{2}{n}-2}dz,
$$
where $K_B^{\kappa n-\frac{n}{2}}$ is given by relation (\ref{2.8}).\\
By homogeneity of $K_B^{\kappa n-\frac{n}{2}}$ (\ref{2.9}), it is clear that the following identity holds:
\begin{eqnarray}\nonumber
 B^{\kappa,n}_{\lambda,e}(x)\,B^{\kappa,n}_{\lambda,e}(y) &=& \frac{1}{n}\int_0^\infty B^{\kappa,n}_{\lambda,e}(z)K_B^{\kappa n-\frac{n}{2}}(| x|^{\frac{1}{n}},| y|^{\frac{1}{n}},z^{\frac{1}{n}})\,z^{2\kappa+\frac{2}{n}-2}dz \\
   &=& \frac{1}{2n}\int_0^\infty B^{\kappa,n}_{\lambda}(z)
   K_B^{\kappa n-\frac{n}{2}}(| x|^{\frac{1}{n}},| y|^{\frac{1}{n}},z^{\frac{1}{n}})\,z^{2\kappa+\frac{2}{n}-2}dz. \label{eveneven}
\end{eqnarray}

We turn now to purely odd products.

Making the same substitution and changing the integration variable (\ref{change}) in (\ref{key1}), we obtain\\

$\displaystyle n^{2n}\lambda^2|xy|^nj_{\kappa n+\frac{n}{2}}(n|\lambda x|^\frac{1}{n})j_{\kappa n+\frac{n}{2}}(n|\lambda y|^\frac{1}{n})$
\begin{eqnarray*}&=&A_{\kappa,n}|\lambda|^{2\kappa+\frac{2}{n}-1}\,\int_0^\infty j_{\kappa n-\frac{n}{2}}(n(|\lambda| z)^\frac{1}{n})K_B^{\kappa n-\frac{n}{2}}\left(n|\lambda x|^\frac{1}{n},n|\lambda y|^\frac{1}{n},n(|\lambda| z)^\frac{1}{n}\right)\\ &\times & C_n^{\kappa n-\frac{n}{2}}\left(\sigma^n_{n^n\lambda x,n^n\lambda y,n^n\lambda z}\right)\,z^{2\kappa+\frac{2}{n}-2}dz,\end{eqnarray*}
where $\displaystyle A_{\kappa,n}=\frac{n!n^{2\kappa-n+1}\left((\kappa n-\frac{n}{2}+1)_n\right)^2}{(2\alpha)_n}$  and $\sigma_{x,y,z}^n$ is given by (\ref{Delta}).\\
The homogeneity of $K_B^{\kappa n-\frac{n}{2}}$ (\ref{2.9}) and $\sigma^n$, leads to\\

$\displaystyle n^{2n}\lambda^2|xy|^nj_{\kappa n+\frac{n}{2}}(n|\lambda x|^\frac{1}{n})j_{\kappa n+\frac{n}{2}}(n|\lambda y|^\frac{1}{n})$
$$=A_{\kappa,n}\,\int_0^\infty j_{\kappa n-\frac{n}{2}}(n(|\lambda| z)^\frac{1}{n})K_B^{\kappa n-\frac{n}{2}}\left(|x|^\frac{1}{n},|y|^\frac{1}{n},z^\frac{1}{n}\right) C_n^{\kappa n-\frac{n}{2}}\left(\sigma^n_{x,y,z}\right)\,z^{2\kappa+\frac{2}{n}-2}dz.$$
By arguing  by evenness and oddness, we deduce, for all, $x,y\in \mathbb{R}^\star$,  \\

$\displaystyle
B^{\kappa,n}_{\lambda,o}(x)\,B^{\kappa,n}_{\lambda,o}(y) $
 \begin{eqnarray}\nonumber&=&\displaystyle\frac{(-1)^nn!\textrm{sgn}(xy)}{n(2\kappa n-n)_n}\int_0^\infty B^{\kappa,n}_{\lambda,e}(z)
 K_B^{\kappa n-\frac{n}{2}}\left(|x|^\frac{1}{n},|y|^\frac{1}{n},z^\frac{1}{n}\right)C_n^{\kappa n-\frac{n}{2}}(\sigma_{x,y,z}^n)\,
 z^{2\kappa-2+\frac{2}{n}}dz \nonumber\\
&=&\frac{1}{2n}\int_{\mathbb{R}}B^{\kappa,n}_\lambda(z)K^{\kappa n-\frac{n}{2}}_{B}\left(| x |^{\frac{1}{n}},| y|^{\frac{1}{n}},
|z|^{\frac{1}{n}}\right)\;\xi_{\kappa,n}(x,y,z)\,|z|^{2\kappa-2+\frac{2}{n}}dz,\label{oddodd}
\end{eqnarray}
where $\xi_{\kappa,n}$ is given by (\ref{ksi}).

We consider now the mixed products.

 Substitute $\alpha=\kappa n-\frac{n}{2}$, $u=n|\lambda x|^\frac{1}{n}$, $v=n| \lambda y|^\frac{1}{n}$,
 perform a change of variable (\ref{change}) in (\ref{nqq}), and use the homogeneity of the Bessel kernel $K_B^{\kappa n-\frac{n}{2} }$
 and $\sigma_{z,x,y}^n$, it derives\\

 $\displaystyle
n^n|\lambda x|j_{\kappa n+\frac{n}{2}}(n|\lambda x |^{\frac{1}{n}})j_{\kappa n-\frac{n}{2}}(n|\lambda y |^{\frac{1}{n}})$
$$=\frac{1}{n}\displaystyle \frac{n!}{(2\kappa n-n)_n}
\int_{0}^{\infty}n^n|\lambda |zj_{\kappa n+\frac{n}{2}}(n|\lambda|^{\frac{1}{n}}z^{\frac{1}{n}})K^{\kappa n-\frac{n}{2}}_{B}
\left(| x |^{\frac{1}{n}},| y|^{\frac{1}{n}},z^{\frac{1}{n}}\right)
C_n^{\kappa n-\frac{n}{2}}(\sigma_{z,x,y}^n)\,z^{2\kappa+\frac{2}{n}-2}dz.
$$
By arguing again  by evenness and oddness, we deduce, for all, $x,y\in \mathbb{R}^\star$,  \\

$\displaystyle B^{\kappa,n}_{\lambda,o}(x)\;B^{\kappa,n}_{\lambda,e}(y)$
\begin{eqnarray}&=& \nonumber \frac{n!}{n(2\kappa n-n)_n}
\int_0^\infty B^{\kappa,n}_{\lambda,o}(z)K^{\kappa n-\frac{n}{2}}_{B}\left(| x |^{\frac{1}{n}},| y|^{\frac{1}{n}},z^{\frac{1}{n}}\right)
\textrm{\textrm{sgn}}(xz)C_n^{\kappa n-\frac{n}{2}}(\sigma_{z,x,y}^n)\,z^{2\kappa+\frac{2}{n}-2}dz\\
&=& \nonumber \frac{n!}{2n(2\kappa n-n)_n}
\int_{\mathbb{R}}B^{\kappa,n}_\lambda(z)K^{\kappa n-\frac{n}{2}}_{B}\left(| x |^{\frac{1}{n}},| y|^{\frac{1}{n}},|z|^{\frac{1}{n}}\right)
\textrm{\textrm{sgn}}(xz)C_n^{\kappa n-\frac{n}{2}}(\sigma_{z,x,y}^n)\,|z|^{2\kappa+\frac{2}{n}-2}dz\nonumber \\
&=&\frac{1}{2n}\int_{\mathbb{R}}B^{\kappa,n}_\lambda(z)K^{\kappa n-\frac{n}{2}}_{B}\left(| x |^{\frac{1}{n}},| y|^{\frac{1}{n}},|z|^{\frac{1}{n}}\right)
\;\xi_{\kappa,n}(z,x,y)\,|z|^{2\kappa-2+\frac{2}{n}}dz.\label{mixed}
\end{eqnarray}
Putting together (\ref{eveneven}), (\ref{oddodd}) and (\ref{mixed})  finishes the proof of Theorem \ref{th1}.
$\hfill\Box$\\
\noindent\textbf{Remark.} As announced in the introduction, the measure $\nu_{x,y}^{\kappa,n}$ is not positive.\\
Indeed, let $0<y<x$, then from (\ref{calK}) and (\ref{Delta}), we obtain
$$\begin{array}{ccc}

  \sigma^n_{x^n,y^n,(x-y)^n}=&-\sigma^n_{(x-y)^n,x^n,y^n}=&-\sigma^n_{(x-y)^n,y^n,x^n}=1 \\
  \sigma^n_{x^n,y^n,(x+y)^n}=&-\sigma^n_{(x+y)^n,x^n,y^n}=&-\sigma^n_{(x+y)^n,y^n,x^n}=1.\end{array}
$$
In the other hand, (\ref{gegen}) leads to
$$C_n^{\kappa n-\frac{n}{2}}(1)=(-1)^nC_n^{\kappa n-\frac{n}{2}}(-1)=\frac{(2\kappa n-n)_n}{n!}.$$
Since $K_B^{\kappa n-\frac{n}{2}}$ is positive, we deduce that
$$\mathcal{K}_{\kappa,2n+1}(x^n,y^n,-(x-y)^n)<0;\quad \mathcal{K}_{\kappa,2n}(x^n,y^n,-(x+y)^n)<0.$$
Hence there exists a neighborhood of $-(x-y)^n$ in supp$ (\nu_{x^n,y^n}^{\kappa,2n+1})$ (rep. of $-(x+y)^n$ in supp$ (\nu_{x^n,y^n}^{\kappa,2n})$)
such that the function
$z\rightarrow \mathcal{K}_{\kappa,2n+1}(x,y,z)$ (resp. $z\rightarrow \mathcal{K}_{\kappa,2n}(x,y,z)$ is strictly negative.

The following theorem gives some properties of the measure $\nu_{x,y}^{\kappa,n}$ as announced in the introduction.
\begin{theorem}\label{th2}
  For $n\in \mathbb{N}^\star$, $\kappa>\frac{n-1}{2n}$ and $x,y\in\mathbb{R}$, we have
  \begin{itemize}
    \item [i)] $\displaystyle \textrm{supp}(\nu_{x,y}^{\kappa,n})(\mathbb{R})\subset I_{x,y}=\left\{z\in \mathbb{R}/\; ||x|^\frac{1}{n}-|y|^\frac{1}{n}|<|z|^\frac{1}{n}<|x|^\frac{1}{n}+|y|^\frac{1}{n}\right\}$.
    \item [ii)] $\displaystyle \nu_{x,y}^{\kappa,n}(\mathbb{R})=1$.
    \item [iii)] $\displaystyle \left\|\nu_{x,y}^{\kappa,n}\right\|\: \leq 4$.
  \end{itemize}
\end{theorem}
\begin{proof} i) is clear.\\
  ii) follows from Theorem \ref{th1} and the fact that $B^{\kappa,n}_0(x)=1$.\\
  iii) From the definition (\ref{main}) of the measure $\nu_{x,y}^{\kappa,n}$, it follows from (\ref{boudednessC})
  $$\left\|\nu_{x,y}^{\kappa,n}\right\|\leq \frac{2}{n}\int_{0}^{+\infty}
  K_B^{\kappa n-\frac{n}{2}}(|x|^\frac{1}{n},|y|^\frac{1}{n},|z|^\frac{1}{n})|z|^{2\kappa+\frac{2}{n}-2}dz=4.$$
  Above we used the positivity of $K_B^{\kappa n-\frac{n}{2}}$, relation (\ref{maj}) and relation (\ref{Kalpha}).\\
  This finishes the proof of Theorem \ref{th2}.
\end{proof}

\section{Convolution structure}

Let us denote by

$\bullet\quad\mathcal{C}_b(\mathbb{R})$ the space of bounded continuous functions on $\mathbb{R}$.

$\bullet\quad\mathcal{C}_c(\mathbb{R})$ the space of continuous functions on $\mathbb{R}$ with compact support.

 Recall that for $n\in \mathbb{N}^\star$ and $\kappa>\frac{n-1}{2n}$, the generalized Hankel transform $\mathcal{F}_{{\kappa,n}}$ is defined by (\ref{Fka}). Its inverse is given
by
$$\label{5.1}
\mathcal{F}_{{\kappa,n}}^{-1}(g)(x)=\mathcal{F}_{{\kappa,n}}(g)((-1)^nx),\quad x\in \mathbb{R}.$$

The generalized Hankel transform $\mathcal{F}_{{\kappa,n}}$ can be expressed in terms of Hankel transform
\begin{equation}\label{5.2}
  \mathcal{H}_{\alpha}(f)(\lambda)=\frac{1}{2^{\alpha-1}\Gamma(\alpha+1)}\int_0^{+\infty}f(t)j_\alpha(t\lambda)t^{2\alpha+1}dt,\quad \lambda\in ]0,+\infty[.
\end{equation}

More precisely:
\begin{proposition}
 Let  $n\in \mathbb{N}^\star$, $ \kappa>\frac{n-1}{n} $, and  $ f \in \mathcal{C}_c(\mathbb{R})$. Then
 $$\mathcal{F}_{\kappa,n}(f)(\lambda)=\frac{1}{2n^{\kappa n-\frac{n}{2}+1}}
 \mathcal{H}_{\kappa n-\frac{n}{2}}(g_n)(|\lambda|^{\frac{1}{n}})+\frac{(-i)^n\lambda}{n!2^{n}n^{\kappa n-\frac{n}{2}+1}}
 \mathcal{H}_{\kappa n-\frac{n}{2}}(J_n(f_o))(|\lambda|^{\frac{1}{n}}),\quad \lambda\in \mathbb{R},$$
 where $g_n$ and $J_n$ are the functions defined on $\mathbb{R}_+$ by
$$g_n(t)=f_e((\frac{t}{n})^{n}) ;\quad J_n(f_o)(t)=\int_s^\infty f_o((\frac{t}{n})^{n})(t^2-s^2)^{n-1}t^{-n+1} dt,$$
and $f_e$ and $f_o$ are the even and odd parts of the function $f$.
\end{proposition}
\begin{proof}
  By making a change of variable and using (\ref{5.2}), we get
\begin{eqnarray}
\mathcal{F}_{\kappa,n}(f_{e})(\lambda)
&=&\left(\frac{2}{n}\right)^{-\left(\kappa n-\frac{n}{2}\right)}\Gamma\left(\kappa n-\frac{n}{2}+1\right)^{-1}\displaystyle\int^{\infty}_{0}f_{e}(x)j_{\kappa n-\frac{n}{2}}(n(|\lambda|x)^{\frac{1}{n}})x^{2\kappa+\frac{2}{n}-2}dx\nonumber\\
&=&\frac{1}{2^{\kappa n-\frac{n}{2}}n^{\kappa n-\frac{n}{2}+1}\Gamma(\kappa n-\frac{n}{2}+1)}\displaystyle\int_{0}^{\infty}f_{e}((\frac{t}{n})^{n}) j_{\kappa n-\frac{n}{2}}(t|\lambda|^{\frac{1}{n}})t^{2\kappa n-n+1}dt\nonumber\\
&=&\frac{1}{2n^{\kappa n-\frac{n}{2}+1}}\mathcal{H}_{\kappa n-\frac{n}{2}}(g_n)(|\lambda|^{\frac{1}{n}}).\label{5.3}
\end{eqnarray}
Proceeding in similar way, we get
$$
\mathcal{F}_{\kappa,n}(f_{o})(\lambda)=\frac{(-i)^n (\frac{n}{2})^{\kappa n+\frac{n}{2}}\lambda}{\Gamma\left(\kappa n+\frac{n}{2}+1\right)}
\displaystyle\int_{0}^{\infty}f_{o}(x)j_{\kappa n+\frac{n}{2}}(n|\lambda x|^{\frac{1}{n}})x^{2\kappa+\frac{2}{n}-1}dx.$$
Sonine's integral formula (\ref{2.16}) and Fubini's theorem lead to
\begin{eqnarray*}
\mathcal{F}_{\kappa,n}(f_{o})(\lambda)&=&C_{\kappa,n}\lambda\;\int_{0}^{\infty}f_o((\frac{t}{n})^{n})
\left[\int_0^1(1-u^2)^{n-1}j_{\kappa n-\frac{n}{2}}(tu|\lambda|^{\frac{1}{n}})u^{2\kappa n-n+1}du\right]t^{2\kappa n-n+1}dt  .\nonumber\\
 &=&C_{\kappa,n}\lambda\;\int_{0}^{\infty}f_o((\frac{t}{n})^{n})\left[\int_0^t(t^2-s^2)^{n-1}j_{\kappa n-\frac{n}{2}}(s|\lambda|^{\frac{1}{n}})s^{2\kappa n-n+1}ds\right]t^{-n+1}dt  .\nonumber\\
 &=&C_{\kappa,n}\lambda\;\int_{0}^{\infty}j_{\kappa n-\frac{n}{2}}(s|\lambda|^{\frac{1}{n}})\left[\int_s^\infty f_o((\frac{t}{n})^{n})(t^2-s^2)^{n-1}t^{-n+1}dt\right]s^{2\kappa n-n+1}ds,
\end{eqnarray*}
where
$ \displaystyle C_{\kappa,n}=\frac{(-i)^n }{n!2^{\kappa n+\frac{n}{2}-1}n^{\kappa n-\frac{n}{2}+1}\Gamma\left(\kappa n-\frac{n}{2}+1\right)}$.\\
Using (\ref{5.2}, it
 derives
\begin{equation}\label{5.4}
 \mathcal{F}_{\kappa,n}(f_{o})(\lambda)= \frac{(-i)^n\lambda}{2^{n}n!n^{\kappa n-\frac{n}{2}+1}}\mathcal{H}_{\kappa n-\frac{n}{2}}(J_n(f_o))(|\lambda|^{\frac{1}{n}}).
\end{equation}
Combining  (\ref{5.3}) and (\ref{5.4}), we obtain the result.
\end{proof}

Ben Said and Al proved in \cite[Theorem 5.1]{BKO} that the generalized Hankel transform $\mathcal{F}_{\kappa,n}$ and its inverse $\mathcal{F}_{\kappa,n}^{-1}$ are topological isomorphisms from $\mathcal{S}(\mathbb{R})$ into itself. They gave a Plancherel's formula
$$\|\mathcal{F}_{\kappa,n}f\|_{\kappa,2}=\|f\|_{\kappa,2}.$$
They also proved that $\mathcal{F}_{\kappa,n}$ can be extended to a topological isomorphism from  $ L^2(\mathbb{R},d\mu_{\kappa,n})$
into itself.

Furthermore, since $B_\lambda^{\kappa,n}$ is bounded by $1$ then we get easily that $\mathcal{F}_{\kappa,n}$ is well defined for $f\in L^1(\mathbb{R},d\mu_{\kappa,n})$ and we have $$\|\mathcal{F}_{\kappa,n}f\|_{\kappa,\infty}\leq\|f\|_{\kappa,1}.
$$
By Riesz-Thorin interpolation theorem we extend the definition of $\mathcal{F}_{\kappa,n}$ for functions $f\in L^p(\mathbb{R},d\mu_{\kappa,n})$, where $1\leq p\leq 2$ and we have the Hausdorff-Young inequality $$\|\mathcal{F}_{\kappa,n}f\|_{\kappa,p'}\leq\|f\|_{\kappa,p}.
$$
\begin{definition}\label{def4.1}
Let $x\in\mathbb{R}$ and  $ f \in \mathcal{C}_b(\mathbb{R}) $. For $n\in \mathbb{N}^\star$ and $\kappa>\frac{n-1}{2n}$, we define the translation operator $ \tau_y^{\kappa,n} $ by
$$
\tau_x^{\kappa,n}f(y)=\displaystyle\int_{\mathbb{R}}f(z)d\nu_{x,y}^{\kappa,n}(z),\quad y\in \mathbb{R},
$$
where $d\nu_{x,y}^{\kappa,n}$ is given by (\ref{main}).
\end{definition}

The following properties hold.
\begin{proposition}
 Let  $n\in \mathbb{N}^\star$, $ \kappa>\frac{n-1}{n} $, $x\in\mathbb{R}$ and  $ f \in \mathcal{C}_b(\mathbb{R}) $. Then
 \begin{itemize}
\item[(i)] $\tau_x^{\kappa,n}f(y)=\tau_y^{\kappa,n}f(x)$.
\item[(ii)] $\tau_0^{\kappa,n}f=f$.
\item[(iii)] $\tau_x^{\kappa,n}\tau_y^{\kappa,n}=\tau_y^{\kappa,n}\tau_x^{\kappa,n}.$\\
If we suppose also that $ f \in \mathcal{C}_c(\mathbb{R}) $, then
\item[(iv)] $ \mathcal{F}_{\kappa,n} \left(\tau_x^{\kappa,n}f\right)(\lambda)=B_\lambda^{\kappa,n}((-1)^nx) \mathcal{F}_{\kappa,n}(f)(\lambda)$.
\item[(v)] $T^{\kappa,n} \tau_x^{\kappa,n}=\tau_x^{\kappa,n}T^{\kappa,n}$,\\
where $$\label{modifiedhankeloperator}
	T^{\kappa,n}f(x)=|x|^{2(1-\frac{1}{n})}\left\lbrace \displaystyle \frac{d^{2}}{dx^{2}}f(x)+\displaystyle \frac{2\kappa}{x}\displaystyle \frac{d}{dx}f(x)-\displaystyle \frac{\kappa}{x^{2}}(1-s)f(x)\right\rbrace,
	$$
	here $sf(x)=f(-x),$ for all $x\in\mathbb{R}.$
\end{itemize}
\end{proposition}
\begin{proof}
  i) follows from the property $\mathcal{K}_{\kappa,n}(x,y,z)=\mathcal{K}_{\kappa,n}(y,x,z)$.\\
  ii) is a consequence of the fact that $B_\lambda^{\kappa,n}(0)=1$.\\
  iii) follows from i).\\
  iv) Let $ f \in \mathcal{C}_c(\mathbb{R}) $, then from Definition \ref{def4.1} and Fubini's theorem, we obtain
  \begin{eqnarray*}
    \mathcal{F}_{\kappa,n}(\tau_x^{\kappa,n}f)(\lambda) &=& \int_{\mathbb{R}} \tau_x^{\kappa,n}f(y)B_\lambda^{\kappa,n}(y)d\mu_{\kappa,n}(y)\\
     &=& \int_{\mathbb{R}}\left[\int_{\mathbb{R}}f(z)\mathcal{K}_{\kappa,n}(x,y,z)d\mu_{\kappa,n}(z)\right] B_\lambda^{\kappa,n}(y)d\mu_{\kappa,n}(y)\\
     &=& \int_{\mathbb{R}}f(z)\left[\int_{\mathbb{R}}\mathcal{K}_{\kappa,n}(x,y,z)B_\lambda^{\kappa,n}(y)d\mu_{\kappa,n}(y)\right]d\mu_{\kappa,n}(z).
  \end{eqnarray*}
  The property $\mathcal{K}_{\kappa,n}(x,y,z)=\mathcal{K}_{\kappa,n}((-1)^nx,z,y)$, gives
  $$\mathcal{F}_{\kappa,n}(\tau_x^{\kappa,n}f)(\lambda)=\int_{\mathbb{R}}f(z)
  \left[\int_{\mathbb{R}}\mathcal{K}_{\kappa,n}((-1)^nx,z,y)B_\lambda^{\kappa,n}(y)d\mu_{\kappa,n}(y)\right]d\mu_{\kappa,n}(z).$$
 Using Theorem \ref{th1}, we see that
  $$\mathcal{F}_{\kappa,n}(\tau_x^{\kappa,n}f)(\lambda)=B_\lambda^{\kappa,n}((-1)^nx)\mathcal{F}_{\kappa,n}(f)(\lambda).$$
v) First, we note that
$$T^{\kappa,n}B_\lambda^{\kappa,n}(y)=-|\lambda|^{\frac{2}{n}}B_\lambda^{\kappa,n}(y)$$
and
$$\int_{\mathbb{R}}T^{\kappa,n}f(y)g(y)d\mu_{\kappa,n}(y)=\int_{\mathbb{R}}f(y)T^{\kappa,n}g(y)d\mu_{\kappa,n}(y).$$
Thus, from iv), we see that
$$\mathcal{F}_{\kappa,n}\left(T^{\kappa,n} \tau_x^{\kappa,n}f\right)(\lambda)=\mathcal{F}_{\kappa,n}\left( \tau_x^{\kappa,n}T^{\kappa,n}f\right)(\lambda)=-|\lambda|^{\frac{2}{n}}B_\lambda^{\kappa,n}\mathcal{F}_{\kappa,n}(f)(\lambda).$$
The assertion follows from the injectivity of the generalized Hankel transform.
\end{proof}
\begin{lemma}
 Let  $n\in \mathbb{N}^\star$ and $ \kappa>\frac{n-1}{2n} $,
 $ 1\leq p \leq \infty $, $f\in L^{p}(\mathbb{R},d\mu_{\kappa,n})$ and $x\in\mathbb{R}$. Then
\begin{equation}\label{Trans}
\|\tau_x^{\kappa,n}(f)\|_{\kappa,p}\leq 4\,\|f\|_{\kappa,p}, \ \ x\in\mathbb{R}. \end{equation}
\end{lemma}
\begin{proof}
 We distinguish the cases:\\
\textbf{Case 1:} $p=\infty$ is obvious.\\
\textbf{Case 2:} If $p=1$, the assertion follows from Fubini-Tonelli's theorem, the property $\mathcal{K}_{\kappa,n}(x,y,z)=\mathcal{K}_{\kappa,n}((-1)^nx,z,y)$ and iii) of Theorem \ref{th2}.\\
\textbf{Case 3:} Let $1<p<+\infty$ and $p'$ denotes the H\"older conjugate exponent of $p$. Then by H\"older inequality, we have
$$|\tau_x^{\kappa,n} f(y)|^p\leq \left(\int_\mathbb{R}|f(z)|^p|\mathcal{K}_{\kappa,n}(x,y,z)|d\mu_{\kappa,n}(z)\right)
\left(\int_\mathbb{R}|\mathcal{K}_{\kappa,n}(x,y,z)|d\mu_{\kappa,n}(z)\right)^\frac{p}{p'}$$
Therefore
$$
\|\tau_{x}^{\kappa,n}f\|^{p}_{\kappa,p}\leq 4^{\frac{p}{p'}}\int_{\mathbb{R}}\int_{\mathbb{R}}|f(z)|^{p}\,|\mathcal{K}_{\kappa,n}(x,y,z)| \,d\mu_{\kappa,n}(z)\,d\mu_{\kappa,n}(y).
$$
Using  again Fubini's theorem and the property $\mathcal{K}_{\kappa,n}(x,y,z)=\mathcal{K}_{\kappa,n}((-1)^nx,z,y)$, we get
\begin{eqnarray*}
\|\tau_{y}^{\kappa,n}f\|^{p}_{\kappa,p}&\leq & 4^{\frac{p}{p'}}\int_{\mathbb{R}}|f(z)|^{p}\,\int_{\mathbb{R}}|\mathcal{K}_{\kappa,n}((-1)^nx,z,y)| \,d\mu_{\kappa,n}(y)\,d\mu_{\kappa,n}(z)\\ &=& 4^{p}\|f\|_{\kappa,p}^p.
\end{eqnarray*}
 Thus, by taking the $\left(1/p\right)^{th}$ power in both sides, we obtain (\ref{Trans}).
\end{proof}

\begin{definition}\label{def4.5}
The convolution product of two suitable functions $ f $ and $ g $ on $ \mathbb{R} $ is defined by
	$$f\star_{\kappa,n} g(x)=\displaystyle\int_{\mathbb{R}}f(y)\,\tau_{x}^{\kappa,n} g((-1)^ny)\;d\mu_{\kappa,n}(y).$$
\end{definition}
It shares the following immediate properties.
\begin{properties}
\begin{itemize}
\item[i)] $f\star_{\kappa,n} g=g\star_{\kappa,n} f$.
    \item[ii)] $\left(f\star_{\kappa,n} g\right)\star_{\kappa,n} h=f\star_{\kappa,n}\left(g \star_{\kappa,n}h\right)$.
    \item[iii)] \textbf{(Young inequality)} For $ p,q,r $ such that $ 1\leq p,q,r\leq \infty $ and $ \frac{1}{p}+\frac{1}{q}-1=\frac{1}{r}, $ and
     for $ f\in L^{p}(\mathbb{R},d\mu_{\kappa,n}) $ and $ g\in L^{q}(\mathbb{R},d\mu_{\kappa,n}), $ the convolution product
     $ f\star_\kappa g $ is a well defined element in $ L^{r}(\mathbb{R},d\mu_{\kappa,n}) $ and
     $$ \|f\star_{\kappa,n} g\|_{\kappa,r}\leq 4\,\|f\|_{\kappa,p}\|g\|_{\kappa,q}. $$
    \end{itemize}
    \end{properties}
    \begin{proof}
 i) By using Fubini's theorem and the property $\mathcal{K}_{\kappa,n}(x,y,z)=\mathcal{K}_{\kappa,n}((-1)^nx,z,y)$, we obtain
 \begin{eqnarray*}
   f\star_{\kappa,n} g(x) &=&\int_{\mathbb{R}}f(y)\left[\int_{\mathbb{R}}g(z)\mathcal{K}_{\kappa,n}(x,(-1)^ny,z)d\mu_{\kappa,n}(z)\right]d\mu_{\kappa,n}(y)\\
    &=& \int_{\mathbb{R}}g(z)\left[\int_{\mathbb{R}}f(y)\mathcal{K}_{\kappa,n}((-1)^nx,z,(-1)^ny)d\mu_{\kappa,n}(y)\right]d\mu_{\kappa,n}(z).
 \end{eqnarray*}
 So using the property $\mathcal{K}_{\kappa,n}((-1)^nx,z,(-1)^ny)=\mathcal{K}_{\kappa,n}(x,(-1)^nz,y)$, we get
 \begin{eqnarray*}
   f\star_{\kappa,n} g(x) &=& \int_{\mathbb{R}}g(z)\left[\int_{\mathbb{R}}f(y)\mathcal{K}_{\kappa,n}(x,(-1)^nz,y)d\mu_{\kappa,n}(y)\right]d\mu_{\kappa,n}(z)\\
  &=& \int_{\mathbb{R}}g(z)\tau_{x}^{\kappa,n}f((-1)^nz)d\mu_{\kappa,n}(z) \\
  &=& g\star_{\kappa,n} f(x).
 \end{eqnarray*}
 ii) is obvious.\\
iii) follows by standard arguments.
\end{proof}

  For every $R > 0$, let us denote by $\mathcal{C}_R(\mathbb{R})$ the space of smooth functions on $\mathbb{R}$
  which are supported in $[-R,R ]$. Then
\begin{proposition}For $f\in \mathcal{D}_{R_1}(\mathbb{R})$ and $g\in \mathcal{D}_{R_2}(\mathbb{R})$, then $f\star_{\kappa,n} g\in \mathcal{D}_{R_1+R_2}(\mathbb{R})$ and  we have
$$\mathcal{F}_{\kappa,n}(f\star_{\kappa,n} g)=\mathcal{F}_{\kappa,n}(f)\mathcal{F}_{\kappa,n}(g).$$
   \end{proposition}
\begin{proof} Using Fubini's theorem, we have\\

$\displaystyle \mathcal{F}_{\kappa,n}\left(f\star_{\kappa,n} g\right)(\lambda)$
\begin{eqnarray*}
  &=& \int_{\mathbb{R}}B_\lambda^{\kappa,n}(x)f\star_{\kappa,n} g (x)d\mu_{\kappa,n}(x)\\
   &=& \int_{\mathbb{R}}f(z)\left[\int_{\mathbb{R}}g(y)\left(\int_{\mathbb{R}}B_\lambda^{\kappa,n}(x)\mathcal{K}_{\kappa,n}(x,(-1)^ny,z)
   d\mu_{\kappa,n}(x)\right)d\mu_{\kappa,n}(y)\right] d\mu_{\kappa,n}(z).
\end{eqnarray*}
Invoking the property $\mathcal{K}_{\kappa,n}(x,(-1)^ny,z)=\mathcal{K}_{\kappa,n}(y,z,x)$ and Theorem \ref{th1}, lead to
\begin{eqnarray*}
 \mathcal{F}_{\kappa,n}\left(f\star_{\kappa,n} g\right)(\lambda) &=& \left(\int_{\mathbb{R}}f(z)B_\lambda^{\kappa,n}(z)d\mu_{\kappa,n}(z)\right)
 \left(\int_{\mathbb{R}}g(y)B_\lambda^{\kappa,n}(y)d\mu_{\kappa,n}(y)\right) \\
   &=& \mathcal{F}_{\kappa,n}(f)(\lambda)\mathcal{F}_{\kappa,n}(g)(\lambda).
\end{eqnarray*}
Which proves the proposition.

\end{proof}

%%%%%%%%%%%%%%%%%%%%%%%%%%%%%%%%%%%%%%%%%%%%%%%%%%%%%%%%%%%%%%%
%\section*{Conflict of interest}
%
%The authors declare that they have no conflict of interest.
%\\
%\\
%Data sharing not applicable to this article as no datasets were  analysed during the current study.

%\begin{acknowledgements}
%\end{acknowledgements}



%\begin{thebibliography}{}
% etc
%\end{thebibliography}

\end{document}